\documentclass[english, 12pt]{amsart}

\usepackage{tikz-cd}
\usepackage{aliascnt}
\usepackage{amsfonts}
\usepackage{mathrsfs}
\usepackage{amsthm}
\usepackage{amsmath,amssymb}
\usepackage[all,poly,knot]{xy}
\usepackage{amsfonts}
\usepackage{color}
\usepackage{hyperref}
\usepackage{hyperref}
\usepackage[hyperpageref]{backref}

\usepackage{comment}
\usepackage{csquotes}
\usepackage{amscd}
\usepackage{mathtools}
\usepackage{dsfont}

\usepackage[charter]{mathdesign}

\DeclareFontFamily{U}{dutchcal}{\skewchar\font=45 }
\DeclareFontShape{U}{dutchcal}{m}{n}{<-> s*[1.0] dutchcal-r}{}
\DeclareFontShape{U}{dutchcal}{b}{n}{<-> s*[1.0] dutchcal-b}{}
\DeclareMathAlphabet{\mathlcal}{U}{dutchcal}{m}{n}
\SetMathAlphabet{\mathlcal}{bold}{U}{dutchcal}{b}{n}

\usepackage{contour}
\usepackage{ulem}
\usepackage{amsmath,calligra,mathrsfs}
\DeclareMathOperator{\sheafhom}{\mathscr{H}\text{\kern -3pt {\calligra\large om}}\,}
\DeclareMathOperator{\sheafend}{\mathscr{E}\text{\kern -3pt {\calligra\large nd}}\,}

\theoremstyle{plain}
\newtheorem{thmx}{Theorem}

\newtheorem{thm}{Theorem}[section]

\newtheorem{lem}[thm]{{Lemma}}

\newtheorem{prop}[thm]{Proposition}

\newtheorem{ques}[thmx]{{Question}}

\newtheorem{defi}[thm]{Definition}

\newtheorem{rmk}[thm]{Remark}
\newtheorem{setup}[thm]{Setup}

\theoremstyle{definition}

\newtheorem{definition-proposition}[subsubsection]{Definition-Proposition}

\newtheorem*{proposition*}{Proposition}

\newtheorem*{conjecture*}{Conjecture}

\newtheorem*{theorem*}{Theorem}

\makeatletter
\newcommand{\namelabel}[1]{%
  \phantomsection
  \renewcommand{\@currentlabel}{#1}
  \label{#1}
}
\makeatother

\newcommand{\thistheoremname}{}
\newtheorem*{genericthm*}{\thistheoremname}
\newenvironment{namedthm*}[1]{\renewcommand{\thistheoremname}{#1}%
	\begin{genericthm*}}
	{\end{genericthm*}}

\theoremstyle{remark}

\numberwithin{equation}{section}

\def\log{\mathrm{log}\,}

\newcommand{\NHC}{\mathrm{NHC}}

\usepackage[top=1in, bottom=1in, left=1in, right=1in]{geometry}

\usepackage{xcolor}
\hypersetup{
	colorlinks,
	linkcolor={red!50!black},
	citecolor={blue!50!black},
	urlcolor={blue!80!black}
}

\contourlength{0.8pt}

\newcommand{\Dol}{\mathrm{Dol}}
\newcommand{\End}{\mathrm{End}}

\begin{document}
\title[Holomorphic isomonodromic deformations of Higgs bundles and related properties]{Holomorphic isomonodromic deformations of Higgs bundles: \\absolute lifting, spectral flatness, and nilpotent rigidity}

\author[Tianzhi Hu]{Tianzhi Hu}
\address{ School of Mathematics and Statistics, Wuhan University, Luojiashan, Wuchang, Wuhan, Hubei, 430072, P.R. China}
\email{hutianzhi@whu.edu.cn}

\author[Mai Shi]{Mai Shi}
\address{ School of Mathematics and Statistics, Wuhan University, Luojiashan, Wuchang, Wuhan, Hubei, 430072, P.R. China}
\email{shimai@whu.edu.cn}

\author[Kang Zuo]{Kang Zuo}
\address{
School of Mathematics and Statistics, Wuhan University, Wuhan, Hubei 430072, P.R. China;
Shanghai Institute for Mathematics and Interdisciplinary Sciences (SIMIS), Shanghai, P.R. China;
Institut f\"ur Mathematik, Johannes Gutenberg-Universit\"at Mainz, Mainz 55099, Germany
}
\email{zuok@uni-mainz.de}
\begin{abstract}
Let $f:X\to S$ be a smooth projective family, and fix a semisimple flat bundle on a fiber $X_0$. Its isomonodromic deformation determines a holomorphic section
$\sigma_{\mathrm{dR}}:S\to M_{\mathrm{dR}}(X/S)$
of the relative de Rham moduli space. Applying the relative non-abelian Hodge correspondence fiberwise gives a section
$\sigma_{\mathrm{Dol}}:S\to M_{\mathrm{Dol}}(X/S)$,
whose value at each $s\in S$ is the Higgs bundle corresponding to the flat bundle $\sigma_{\mathrm{dR}}(s)$. Unlike $\sigma_{\mathrm{dR}}$, the section $\sigma_{\mathrm{Dol}}$ is in general only real analytic.
 We study the geometric consequences of its holomorphicity along a complex analytic subvariety. First, we prove an absolute lifting theorem: the relative isomonodromic Higgs bundle admits an absolute Higgs lift, and the fiberwise harmonic metrics can be modified to solve the Hitchin--Simpson equation on the total space. Second, we prove that the spectral one-form on every resolved irreducible component of the relative spectral scheme is Gauss--Manin flat. As an application, we prove the nilpotent rigidity conjecture of Hu--Sun--Yang--Zuo: if the initial Higgs field is nilpotent, then the Higgs field remains nilpotent along the entire holomorphic isomonodromic locus.
\end{abstract}

\subjclass[2010]{14D20,14C30}
\keywords{}

\maketitle

\setcounter{tocdepth}{1}
\tableofcontents

\section{Introduction}
Let $f:X\longrightarrow S$ be a smooth projective family of smooth
projective varieties over a complex manifold $S$. Let $M_{\mathrm{dR}}(X/S)$ be the relative moduli space of semisimple flat bundles and $M_{\mathrm{Dol}}(X/S)$ be the relative moduli space of polystable Higgs bundles with vanishing Chern classes, both introduced in \cite{SimpsonI,SimpsonII}. Let
\[
\NHC:M_{\mathrm{dR}}(X/S)\xrightarrow{\ \sim\ }M_{\mathrm{Dol}}(X/S)
\]
be non-abelian Hodge correspondence, developed mainly by \cite{Cor,Don,Hitchin,Simpson92,UY}. In \cite{Don,UY}, Donaldson and Uhlenbeck-Yau solve the Hermitian--Yang--Mills equation and establish this correspondence for unitary flat bundles. In \cite{Cor,Hitchin,Simpson92}, Corlette, Hitchin and Simpson solve the harmonic metric equations, equivalently the Hitchin--Simpson equations, and establish this correspondence for any semisimple flat bundle.
 
 We aim to use the non-abelian Hodge theory to study the isomonodromic deformation of a local system on a projective manifold. Hence we make the following setup.
\begin{setup}\label{setup:isomonodromic}
 Fix a point
$0\in S$ and a semisimple flat bundle $(V_0,\nabla_0)$ on $X_0$. Its
isomonodromic deformation defines a (multi-valued) holomorphic section
\[
\sigma_{\mathrm{dR}}:S\to M_{\mathrm{dR}}(X/S)\quad\text{with}\quad \sigma_{\mathrm{dR}}(0)=[(V_0,\nabla_0)]\in M_{\mathrm{dR}}(X_0).
\]
 Then we define the real analytic
section
\[
\sigma_{\mathrm{Dol}}:=\NHC\circ\sigma_{\mathrm{dR}}:
S\to M_{\mathrm{Dol}}(X/S),
\]
which is called the \textbf{isomonodromic deformation of a Higgs bundle}.
\end{setup}

Unlike the de Rham section \(\sigma_{\mathrm{dR}}\), the section \(\sigma_{\mathrm{Dol}}\) is in general not holomorphic. Biswas--Heller--Heller \cite{BHH}, building on the work of To\v{s}i\'c \cite{Tosic}, first observed this phenomenon. More recently, Fu-Sheng \cite{FS} and Hu--Sun--Zuo \cite{HSZ} introduced a non-abelian Kodaira--Spencer map that measures the failure of holomorphicity.

The holomorphicity of \(\sigma_{\mathrm{Dol}}\) is closely related to special geometric properties of the isomonodromic deformation. For the universal family over Teichm\"uller space, To\v{s}i\'c \cite{Tosic} and Collier--Toulisse--Wentworth \cite{CTW} related holomorphicity to the constancy of the associated energy functional, while non-holomorphicity is reflected in its strict plurisubharmonicity.

In another direction, Hu--Sun--Yang--Zuo proved that, when the initial flat bundle underlies a polarized complex variation of Hodge structures, the non-abelian Noether--Lefschetz locus is precisely the maximal complex analytic locus on which \(\sigma_{\mathrm{Dol}}\) becomes holomorphic \cite[Theorem~1.7]{HSYZ}.

The purpose of this paper is to study further geometric consequences of the holomorphicity of \(\sigma_{\mathrm{Dol}}\).

\subsection{Holomorphicity and absolute lifting}

We first consider an absolute lifting problem. In the notation of Setup~\ref{setup:isomonodromic}, let \(U\subset S\) be a complex analytic subvariety through \(0\) such that the restricted section $
\sigma_{\mathrm{Dol}}|_U
$ is holomorphic, and write \(X_U=X\times_S U\). The restriction of \(\sigma_{\mathrm{Dol}}\) determines a holomorphic family of relative Higgs bundles on \(X_U/U\). It is natural to ask whether this relative Higgs field is induced by a Higgs field on the total space \(X_U\), or equivalently, whether it admits an \textbf{absolute lift}. 

There is a global analogue of this problem. For an irreducible representation with finite mapping class group orbit (MCG-finite for short), Landesman--Litt show that, after passing to a suitable finite \'etale cover, the representation extends from one fiber to a local system on the total family \cite[Corollary~2.3.5]{LL}. When the total space is smooth quasi-projective, Mochizuki's existence theorem then produces the corresponding absolute harmonic bundle. Thus, MCG-finiteness provides a global mechanism for absolute lifting. For further results on MCG-finite representations and related non-abelian fixed-part phenomena, see \cite{LL,EK} and the references therein.

From this perspective, the holomorphicity of \(\sigma_{\mathrm{Dol}}\) may be viewed as a local Dolbeault analogue of the MCG-finite condition. Our first main theorem establishes the corresponding local lifting property.

\begin{thmx}[=Theorem~\ref{thm:absolute-lifting}]\label{thm_A}
Let \(U\) be a simply connected complex manifold and assume that
\(\sigma_{\mathrm{Dol}}|_U\) is holomorphic. Then the relative isomonodromic Higgs bundle on \(X_U/U\) admits an absolute Higgs lift on \(X_U\). Moreover, the fiberwise harmonic metrics can be chosen so that the induced metric and the absolute Higgs field solve the Hitchin--Simpson equation on the total space \(X_U\).
\end{thmx}

If the total harmonic metric exists, then it induces a global harmonic bundle and $\sigma_{\Dol}|_U$ is automatically holomorphic. Thus for a projective family over a simply connected base, we establish a necessary and sufficient condition for the existence of solutions to the Hitchin-Simpson equation. Unlike all prior existence results, ours does not require the total space to be quasi-projective. There is also a finite-order analogue. Over the Artin thickenings considered in Section~\ref{subsec:deformation-theory}, finite-order holomorphicity of the isomonodromic Dolbeault deformation yields an absolute integrable Higgs lift; see Theorem~\ref{thm:truncated-absolute-lifting}.

Our proof of Theorem~\ref{thm_A} mainly builds on the first-order deformation approach developed in \cite{HSZ,HSYZ} and the existence of the fiberwise harmonic metric. By suitably modifying the fiberwise harmonic metric, we construct the absolute lift explicitly and obtain a global Higgs bundle on the total space.

\subsection{Flatness of the spectral one-form}

Theorem~\ref{thm_A} suggests a second question.  The eigenvalues of a Higgs field are encoded by its spectral one-form on its spectral scheme.  If the relative Higgs field comes from an absolute Higgs field on the total space, one expects this spectral one-form to behave like the restriction of a closed absolute one-form.  This leads naturally to the question whether its fiberwise cohomology class is flat under the Gauss--Manin connection (defined in detail in Definition~\ref{def_relative_s_oneform}). Our second main result gives an affirmative answer to this question.
\begin{thmx}[=Theorem~\ref{thm:spectral-flatness}]\label{thm_B}
Let \(U\subset S\) be an irreducible complex analytic subvariety such that
\(\sigma_{\mathrm{Dol}}|_U\) is holomorphic.  Then the relative spectral one-form of the isomonodromic Higgs family is flat.  More precisely, after resolving each irreducible component of the reduced relative spectral scheme, the cohomology class of the induced spectral one-form is Gauss--Manin flat.\end{thmx}

We prove Theorem~\ref{thm_B} by combining the first-order deformation approach developed in \cite{HSZ,HSYZ} with the observation that the spectral one-form can be expressed as the trace of the Higgs field on the corresponding generalized eigenspace. The use of generalized eigenspaces of a Higgs field over a spectral covering already appears in Zuo's work on the factorization of non-rigid representations \cite{Zuo94}; see also \cite[\S4.3]{ZuoBook} for a detailed treatment.

\subsection{Nilpotent rigidity}
Hu--Sun--Yang--Zuo formulated the following rigidity conjecture: if the initial Higgs field is nilpotent and \(\sigma_{\mathrm{Dol}}\) is holomorphic along a complex analytic subvariety through the initial point, then the entire restricted family should remain nilpotent \cite[Conjecture~1.9]{HSYZ}. Using their higher-order deformation theory, they proved this when the initial Higgs bundle is generically regular nilpotent \cite[Theorem~1.10]{HSYZ}, and conjectured that the generic regularity assumption should be unnecessary.

Using Theorem~\ref{thm_B}, we give a complete proof of the above nilpotent rigidity conjecture.

\begin{thmx}[=Theorem~\ref{thm:nilpotence-rigidity}]\label{thm_C}
Let \(U\subset S\) be an irreducible complex analytic subvariety through \(0\), and assume that
\(\sigma_{\mathrm{Dol}}|_U\) is holomorphic. If the initial Higgs field \(\theta_0\) is nilpotent, then every Higgs bundle represented by \(\sigma_{\mathrm{Dol}}(u)\), \(u\in U\), is nilpotent.
\end{thmx}

This still leaves the following natural characterization problem for the isomonodromic deformation of a nilpotent Higgs bundle.

\begin{ques}
How can one characterize the locus on which \(\sigma_{\mathrm{Dol}}\) remains nilpotent, assuming that the initial Higgs bundle is nilpotent?
\end{ques}

We remark that this nilpotent locus appears to be real analytic in general.

\vspace{.5cm}

\noindent\textbf{AI Declaration:} The authors proposed the research questions and ideas of this paper. AI tools were used for calculations, proving intermediate results, and polishing the writing. All
mathematical statements have been independently checked by the authors, who take full
responsibility for the content.

\section{Isomonodromic deformation theory and absolute lifting}
\label{sec:deformation-lifting}
We work throughout with the isomonodromic Dolbeault section introduced in
Setup~\ref{setup:isomonodromic}.  We first recall its first-order deformation
theory and higher order deformation theory established in \cite[Theorem~A and Proposition~4.2]{HSZ} and \cite[Section~3.3 and Theorem~5.1]{HSYZ}. We then
use these results to prove the absolute lifting Theorem~\ref{thm:absolute-lifting} on a simply connected base and the absolute lifting Theorem~\ref{thm:truncated-absolute-lifting} on an Artin base. 

\subsection{First-order and higher order isomonodromic deformation theory}
\label{subsec:deformation-theory}
In the notation of Setup~\ref{setup:isomonodromic}, fix $0\in S$ and let
$X_0$ be the corresponding fiber.  For an
integer $n\ge1$, let $A_n:=\mathbb C[t]/(t^{n+1})$ be an Artin ring 
and 
\[
\gamma:\operatorname{Spec}A_n\longrightarrow S,
\qquad \gamma(0)=0,
\]
be an $n$-th order germ of $S$ at $0\in S$.  Write
$(E_0,\theta_0):=\sigma_{\mathrm{Dol}}(0),$
and let $h_0$ be a harmonic metric on $(E_0,\theta_0)$.  We denote by
$D_{h_0}=D_{h_0}^{1,0}+\bar\partial_0$ the corresponding Chern
connection. We aim to study the $n$-th truncation of this section along $\gamma$.

On the fixed differentiable manifold $X_0$, pulling back the family $X/S$ by $\gamma$ induces a $n$-th order deformation of
complex structure of $X_0$, denoted by $X_n\to\operatorname{Spec}A_n$. This higher order deformation is represented by a truncated Maurer--Cartan element
\[
\eta=\sum_{k=1}^nt^k\eta_k
\in\mathcal A^{0,1}(T_{X_0})\otimes(t),
\qquad
\bar\partial\eta+\frac12[\eta,\eta]=0.
\]
We write $P'_{\eta}$ and $P''_{\eta}$ for the relative type projections
associated with this deformation. In particular, for $\alpha\in\Omega^{1,0}(X_0)$ and
$\bar\alpha\in\Omega^{0,1}(X_0)$, one has
\[
P'_{\eta}(\alpha)
=
\alpha-\sum_{i=1}^n t^i\eta_i\lrcorner\alpha,
\qquad
P''_{\eta}(\bar\alpha)
=
\bar\alpha-\sum_{i=1}^n \bar t^{\,i}\bar\eta_i\lrcorner\bar\alpha.
\]

Let $h_t$ be the $n$-th truncated family of harmonic metrics attached to
the isomonodromic deformation, and denote the pure holomorphic part of $h_0^{-1}h_t$ by
\begin{equation}\label{eq:G-pure-holomorphic}
G_t:=\operatorname{id}+\sum_{k=1}^nt^kg_k.
\end{equation}

We first take $n=1$ and state the first-order
deformation theory of $\sigma_{\mathrm{Dol}}$ by Hu--Sun--Zuo \cite[Theorem~A and Proposition~4.2]{HSZ}.

\begin{thm}[First-order obstruction and normalization]\label{thm:first-order-normalization}
Let $n=1$.  Then $\sigma_{\mathrm{Dol}}\circ\gamma$ is holomorphic on
$\operatorname{Spec}A_1$ if and only if
\[
\theta_{0,*}([\eta_1])=0
\quad\text{in}\quad
\mathbb H^1\!\left(
X_0,(\operatorname{End}E_0,\operatorname{ad}(\theta_0))
\right).
\]
If these equivalent conditions hold, then we have
\begin{equation}\label{eq:g1-holomorphicity}
[\theta_0,g_1]=0,
\qquad
\eta_1\lrcorner\theta_0
=\frac12\bar\partial_0g_1,
\end{equation}
and there exists a first-order gauge transformation $\mathscr U_1$ such that
\begin{equation}\label{eq:first-order-higgs}
\mathscr U_1^{-1}\theta_t\mathscr U_1
=P'_{t\eta_1}\!\left(
\theta_0-\frac t2D_{h_0}^{1,0}g_1
\right)
\pmod{t^2}.
\end{equation}
\end{thm}

Now we state the higher order deformation theory of the $\sigma_{\mathrm{Dol}}$. For this, we introduce some notions in \cite[Section~3.3]{HSYZ}. For $k\ge1$, let
\[
\operatorname{Comp}(k)
:=\left\{I=(i_1,\ldots,i_N)\,\middle|\,
N\ge1,\ i_a>0,\ i_1+\cdots+i_N=k\right\}.
\]
For $I=(i_1,\ldots,i_N)$ put $\ell(I)=N$.  Set $p_{(k)}=1$ and, for
$N\ge2$,
\[
p_{(i_1,\ldots,i_N)}
:=\prod_{a=2}^{N}\frac{i_a}{i_1+\cdots+i_a}.
\]
For $I=(i_1,\ldots,i_N)$ define
\begin{equation*}
b_I
:=\frac1{2^N}\sum_{j=0}^{N}
 p_{(i_1,\ldots,i_j)}
 p_{(i_N,\ldots,i_{j+1})},
\end{equation*}
with the endpoint terms understood in the usual way. Let $\{g_i\}_{i=1}^n$ be all pure-holomorphic variations of the harmonic metric defined in \eqref{eq:G-pure-holomorphic}. The smooth endomorphisms
$x_k$ are determined recursively by
\begin{equation}\label{eq:xk-triangular}
g_k^{\star_{h_0}}
=\sum_{I\in\operatorname{Comp}(k)}b_Ix_I,
\qquad
x_I:=x_{i_1}\cdots x_{i_N}.
\end{equation}

Finally set
\begin{equation}\label{eq:Y-HSYZ-definition}
y_m:=x_m^{\star_{h_0}},
\qquad
Y:=\sum_{m=1}^nt^my_m,
\qquad
\mathfrak E:=t\cdot\frac{d}{d t}.
\end{equation}

For arbitrary $n$, the same finite-order construction gives the
higher-order normalization of Hu--Sun--Yang--Zuo; see in particular
\cite[Section~5.1, especially Theorem~5.1]{HSYZ}.

\begin{thm}[Higher-order normalization]\label{thm:higher-order-normalization}
Assume that $\sigma_{\mathrm{Dol}}\circ\gamma$ is
holomorphic on $\operatorname{Spec}A_n$, equivalently that the induced
isomonodromic Higgs deformation
$(\mathcal E,\bar\partial_t,\theta_t)$ is holomorphic on $X_n$.  Let
$Y$ be defined by
\eqref{eq:xk-triangular}--\eqref{eq:Y-HSYZ-definition}.  Then there exists
a gauge transformation $\mathscr U\in\mathcal A^0(\End\mathcal E)\otimes\mathbb C[t,\bar t]/(t,\bar t)^{n+1} $ such that
\begin{equation*}
\mathscr U^{-1}\theta_t\mathscr U
=P'_{\eta}\!\left(
\theta_0-\frac12D_{h_0}^{1,0}Y
\right).
\end{equation*}
Moreover, in the same gauge,
\begin{equation*}
\mathscr U^{-1}\bar\partial_t\mathscr U
=\pi''_{\eta}D_{h_0}
+P''_{\eta}\!\left(
-\frac12[\theta_0^{\star_{h_0}},Y]
\right),
\end{equation*}
where $\pi''_{\eta}D_{h_0}$ is the $(0,1)$-part of the smooth operator with respect to the complex structure $X_n$. For $n=1$, this recovers
the gauge-equivalent normal form in Theorem~\ref{thm:first-order-normalization}.
\end{thm}

\subsection{Absolute lifting assuming holomorphicity}
\label{subsec:absolute-lifting}
Let $U\subset S$ be any complex analytic subvariety, and set
\[
X_U:=X\times_S U,
\qquad
p:X_U\longrightarrow U.
\]
For each $u\in U$, write
\[
[(E_u,\theta_u)]:=\sigma_{\mathrm{Dol}}(u).
\]
We denote by $(E_U,\theta_U)$ the resulting real analytic family of Higgs
bundles on $X_U$.

When each $(E_u,\theta_u)$ is stable, let
$h=(h_u)_{u\in U}$ denote a real analytic choice of fiberwise harmonic
metrics.  This choice is not canonical: by stability, each $h_u$ is unique
only up to multiplication by a positive scalar.  Thus $h$ may be replaced by
$c\,h$ for any real analytic function $c:U\to \mathbb R_{>0}$.  

\begin{defi}
We say that the real analytic family $(E_U,\theta_U)$ is \textbf{absolutely
liftable} if there exists a holomorphic Higgs bundle
$(\widetilde E_U,\widetilde\theta)$ on $X_U$, with
\[
\widetilde\theta\in
H^0\!\left(X_U,
\operatorname{End}(\widetilde E_U)\otimes\Omega^1_{X_U}\right),
\qquad
\widetilde\theta\wedge\widetilde\theta=0,
\]
such that its relative restriction to $X_U/U$ is isomorphic to
$(E_U,\theta_U)$.  We call $(\widetilde E_U,\widetilde\theta)$ an
\textbf{absolute lift} of $(E_U,\theta_U)$.
\end{defi}

\begin{rmk}
Note that with this definition, absolute liftability of $(E_U,\theta_U)$ implies that
\(\sigma_{\mathrm{Dol}}|_U\) is holomorphic.  
\end{rmk}

\begin{thm}[Absolute lifting on a holomorphic isomonodromic submanifold]
\label{thm:absolute-lifting}
Assume that $U$ is a simply connected complex manifold and $\sigma_{\mathrm{Dol}}|_U$ is
holomorphic. Let $\nabla$ be the flat connection of the whole isomonodromic family. There is a suitable smooth choice of fiberwise harmonic
metric $h$ such that, if $\nabla^{\star_h}$ denotes the $h$-adjoint
connection and on $X_U$ we set
\[
D_h:=\frac12(\nabla+\nabla^{\star_h}),
\qquad
\Psi:=\frac12(\nabla-\nabla^{\star_h}),
\qquad
\widetilde\theta:=\Psi^{1,0},
\]
then $\widetilde\theta$ is an absolute Higgs field on $X_U$ and satisfies
\begin{enumerate}
\item its image in
$\operatorname{End}(E_U)\otimes\Omega^1_{X_U/U}$ is the given relative
Higgs field $\theta_U$; 
\item $D_h^{0,1}\widetilde\theta=0$;
\item $\widetilde\theta\wedge\widetilde\theta=0$.
\end{enumerate}
In particular, the relative holomorphic Higgs field admits an absolute lift.
\end{thm}

\begin{proof}
\noindent\textbf{Step 1: the stable case.}
Assume first that $(E_u,\theta_u)$ is stable for every $u\in U$. Fix any point in $U$, denoted by $0\in U$. In a local $\nabla$-flat frame we have $\nabla=d$ and $\nabla^{\star_h}=d+H^{-1}dH,$ so
\begin{equation}\label{eq:absolute-flat-frame}
\Psi=-\frac12H^{-1}dH,
\qquad
\widetilde\theta=-\frac12\bigl(H^{-1}dH\bigr)^{1,0}.
\end{equation}
On the fiber $X_0$ this is the usual harmonic-bundle decomposition; hence (1) holds. 

\medskip

Let $v\in T_0^{1,0}U$ and choose a
holomorphic curve
$\gamma:(\Delta,0)\to(U,0)$ with
$\gamma_*(\partial/\partial t)=v$.  Let $\eta_v$ be the Kodaira--Spencer class of $X_U/U$ along $v$ and write $g_v$ as the first-order variation of $H_0^{-1}H$ along $v$ as in \eqref{eq:G-pure-holomorphic}.  Since
$\sigma_{\mathrm{Dol}}|_U$ is holomorphic,
Theorem~\ref{thm:first-order-normalization} gives
\begin{equation}\label{eq:absolute-first-order-identities}
[\theta_0,g_v]=0,
\qquad
\bar\partial_0g_v=2\eta_v\lrcorner\theta_0,
\end{equation}
and, after a first-order gauge transformation,
\begin{equation*}
\theta_t
=P'_{t\eta_v}\!\left(
\theta_0-\frac t2D_{h_0}^{1,0}g_v
\right)+O(t^2).
\end{equation*}
Choose a $C^\infty$ splitting of $T^{1,0}X_U$ near $X_0$. By \eqref{eq:absolute-flat-frame}, there is a horizontal lift $V\in T^{1,0}X_U$ of $v$ such that 
\begin{equation}\label{eq:absolute-horizontal-value}
\widetilde\theta(V)=-\frac12g_v.
\end{equation}

We now prove (3) by checking the three types of pairs determined by this
splitting.

\smallskip
\noindent\emph{Horizontal--horizontal.}
Let $v,w\in T_0^{1,0}U$, with horizontal lifts $V,W$.  By
\eqref{eq:absolute-horizontal-value},
\[
(\widetilde\theta\wedge\widetilde\theta)(V,W)
=\frac14[g_v,g_w].
\]
It remains to show that $[g_v,g_w]=0$.  By
\eqref{eq:absolute-first-order-identities} and the Jacobi identity,
\[
[\theta_0,[g_v,g_w]]=0.
\]
Moreover,
\[
\begin{aligned}
\bar\partial_0[g_v,g_w]
&=[\bar\partial_0g_v,g_w]+[g_v,\bar\partial_0g_w]\\
&=2[\eta_v\lrcorner\theta_0,g_w]
  +2[g_v,\eta_w\lrcorner\theta_0]\overset{\eqref{eq:absolute-first-order-identities}}{=}0.
\end{aligned}
\]
  Hence $[g_v,g_w]$ is a holomorphic Higgs
endomorphism of the stable Higgs bundle $(E_0,\theta_0)$, so it is scalar.
Its trace is zero because it is a commutator, and therefore
$[g_v,g_w]=0$. 

\smallskip
\noindent\emph{Vertical--horizontal.}
If $\xi\in T^{1,0}X_0$ is vertical, then
\[
(\widetilde\theta\wedge\widetilde\theta)(\xi,V)
=-\frac12[\theta_0(\xi),g_v]=0
\]
by \eqref{eq:absolute-first-order-identities} and \eqref{eq:absolute-horizontal-value}.  Thus every mixed
vertical--horizontal component vanishes.

\smallskip
\noindent\emph{Vertical--vertical.}
For vertical vectors $\xi_1,\xi_2$,
\[
(\widetilde\theta\wedge\widetilde\theta)(\xi_1,\xi_2)
=(\theta_0\wedge\theta_0)(\xi_1,\xi_2)=0.
\]

These three cases exhaust $\wedge^2T^{1,0}X_U$ at points of $X_0$; since
$0\in U$ was arbitrary, we obtain (3).

\medskip

It remains to prove (2).  Both $\nabla=D_h+\Psi$ and
$\nabla^{\star_h}=D_h-\Psi$ are flat, hence
\[
D_h\Psi=0,
\qquad
F_{D_h}+\Psi\wedge\Psi=0.
\]
Since $\Psi^{0,1}=\widetilde\theta^{\star_h}$, taking $h$-adjoints in (3) gives
$\Psi^{0,1}\wedge\Psi^{0,1}=0$.  The $(0,2)$-part of the second flatness
identity therefore gives $F_{D_h}^{0,2}=0$, so $D_h^{0,1}$ is an absolute
Dolbeault operator.

We check $D_h^{0,1}\widetilde\theta=0$ componentwise with respect to the same
vertical--horizontal splitting.

\smallskip
\noindent\emph{Vertical--vertical.}
On $X_0$ the operator $D_h^{0,1}$ restricts to $\bar\partial_0$ and
$\widetilde\theta$ restricts to $\theta_0$.  Hence this component is
$\bar\partial_0\theta_0=0$.

\smallskip
\noindent\emph{Mixed components.}
There are two mixed types, which we treat separately.

First let $\bar\xi\in T^{0,1}X_0$ be vertical and let $V$ be the chosen
horizontal lift of $v\in T_0^{1,0}U$.  Since $\widetilde\theta$ is of type
$(1,0)$, the covariant exterior derivative gives, at points of $X_0$,
\[
\bigl(D_h^{0,1}\widetilde\theta\bigr)(\bar\xi,V)
 =D_{h,\bar\xi}\bigl(\widetilde\theta(V)\bigr)
  -\widetilde\theta\bigl([\bar\xi,V]^{1,0}\bigr).
\]
The first term is determined by the horizontal value
\eqref{eq:absolute-horizontal-value}: because
$D_h^{0,1}|_{X_0}=\bar\partial_0$,
\[
D_{h,\bar\xi}\bigl(\widetilde\theta(V)\bigr)
 =-\frac12\bigl(\bar\partial_0g_v\bigr)(\bar\xi).
\]
For the second term, with the convention for the Kodaira--Spencer tensor used
in the projection $P'_{t\eta_v}$, the vertical $(1,0)$-part of the bracket is $[\bar\xi,V]^{1,0}_{\mathrm{vert}}=-\eta_v(\bar\xi).$ We therefore obtain
\[
-\widetilde\theta\bigl([\bar\xi,V]^{1,0}\bigr)
 =\theta_0\bigl(\eta_v(\bar\xi)\bigr)
 =\bigl(\eta_v\lrcorner\theta_0\bigr)(\bar\xi).
\]
Consequently,
\[
\bigl(D_h^{0,1}\widetilde\theta\bigr)(\bar\xi,V)
 =\bigl(\eta_v\lrcorner\theta_0
   -\frac12\bar\partial_0g_v\bigr)(\bar\xi)\overset{\eqref{eq:absolute-first-order-identities}}=0.
\]

By comparing the $(1,1)$-part of $D_h\Psi=0$, we have \begin{align}\label{eq_skewsymmetric}D_h^{0,1}\widetilde\theta=-\bigl(D_h^{0,1}\widetilde\theta\bigr)^{\star_h}.
\end{align}
Thus for the other mixed type, let $\bar V$ be a horizontal $(0,1)$-lift of
$\bar v\in T_0^{0,1}U$ and let $\xi\in T^{1,0}X_0$ be vertical. We have
\[
\bigl(D_h^{0,1}\widetilde\theta\bigr)(\bar V,\xi)
 =\bigl((D_h^{0,1}\widetilde\theta)(\bar\xi,V)\bigr)^{\star_h}=0.
\]
Thus both mixed components of $D_h^{0,1}\widetilde\theta$ vanish.

\smallskip
\noindent\emph{Horizontal--horizontal.}
Finally take $v,w\in T_0^{1,0}U$, with horizontal lifts $V,W$.  We first show
that the horizontal--horizontal component of $D_h^{0,1}\widetilde\theta$ is
necessarily scalar. 
 
 Since $F_{D_h}^{0,2}=0$, the operator $D_h^{0,1}$ is
integrable on $\operatorname{End}(E_U)$, and hence $D_h^{0,1}\bigl(D_h^{0,1}\widetilde\theta\bigr)=0.$ Let $\bar\xi$ be a vertical $(0,1)$-vector on $X_0$.  Since the base point
$0\in U$ was arbitrary, the vertical--vertical and both mixed components of
$D_h^{0,1}\widetilde\theta$ vanish on the neighborhood under consideration.
Evaluating the preceding identity on $(\bar\xi,\bar V,W)$ therefore gives
\[
\bar\partial_0\!\left(
  \bigl(D_h^{0,1}\widetilde\theta\bigr)(\bar V,W)
\right)=0.
\]
Thus
$\bigl(D_h^{0,1}\widetilde\theta\bigr)(\bar V,W)$ is a holomorphic
endomorphism of $E_0$.

Next apply $D_h^{0,1}$ to
$\widetilde\theta\wedge\widetilde\theta=0$.  By the Leibniz rule,
\[
0=
\bigl(D_h^{0,1}\widetilde\theta\bigr)\wedge\widetilde\theta
-\widetilde\theta\wedge
 \bigl(D_h^{0,1}\widetilde\theta\bigr).
\]
Evaluating on $(\bar V,W,\xi)$, where
$\xi\in T^{1,0}X_0$ is vertical, and using again the vanishing of the mixed
components, we obtain
\[
\left[
  \bigl(D_h^{0,1}\widetilde\theta\bigr)(\bar V,W),
  \theta_0(\xi)
\right]=0.
\]
Hence $\bigl(D_h^{0,1}\widetilde\theta\bigr)(\bar V,W)$ is a holomorphic
Higgs endomorphism of the stable Higgs bundle $(E_0,\theta_0)$, and therefore
is scalar.  Since the mixed components vanish, this scalar is independent of
the chosen horizontal lifts.  Consequently there is a $(1,1)$-form
$\beta$ on $U$ such that
\[
D_h^{0,1}\widetilde\theta=p^*\beta\,\operatorname{id}.
\]

It remains to use the scalar freedom of the harmonic metric.  In the local
$\nabla$-flat frame, taking traces in
$D_h^{0,1}\widetilde\theta$ gives
\[
\operatorname{tr}\bigl(D_h^{0,1}\widetilde\theta\bigr)
 =\bar\partial\,\operatorname{tr}(\widetilde\theta)
 =-\frac12\bar\partial\partial\log\det H.
\]
The right-hand side is $d$-closed.  If $r=\operatorname{rank}E$, the scalar
identity above therefore yields
\[
r\,p^*\beta
 =-\frac12\bar\partial\partial\log\det H,
\]
so $\beta$ is a closed $(1,1)$-form on $U$.  Moreover by \eqref{eq_skewsymmetric}, $\beta$ is a pure-imaginary $(1,1)$ form.

Choose a relatively ample line bundle
on $X_U$ with a real closed Chern form $\omega$ that is K\"ahler
on each fiber. Let $N$ be the dimension of the fiber of $X_U/U$ and then $p_*(\omega^N)>0$ is a constant function on $U$.
In local $\nabla$-flat frames, $q:=\log\det H$ differs only by
constants, so $dq$ defines a global closed real one-form
on $X_U$. Consequently,
\[
\lambda:=\frac{1}{p_*(\omega^N)}p_*(dq\wedge\omega^N)
\]
is a closed real one-form on $U$. Since $U$ is simply connected,
$\lambda=df$ for some $f\in C^\infty(U,\mathbb R)$.

Using $\operatorname{tr}\widetilde\theta=-\frac12\partial q$,
we obtain
\[
\partial_U f=-\frac{2}{p_*(\omega^N)}
p_*(\operatorname{tr}\widetilde\theta\wedge\omega^N).
\]
Taking the trace of
$D_h^{0,1}\widetilde\theta=p^*\beta\,\operatorname{id}$
and applying fiber integration gives
\[
r\beta
=\frac{1}{p_*(\omega^N)}\bar\partial_U
p_*(\operatorname{tr}\widetilde\theta\wedge\omega^N)
=-\frac12\bar\partial_U\partial_U f.
\]
Thus $\varphi:=-f/r$ satisfies
\[
\beta=\frac12\bar\partial_U\partial_U\varphi.
\]

We now use the scalar normalization of $h$:
replace $h$ by $e^{\varphi}h$ and continue to denote the normalized metric by
$h$.  Since $\varphi$ is pulled back from $U$, in the same flat frame
\[
\widetilde\theta\longmapsto
\widetilde\theta-\frac12p^*(\partial_U\varphi)\,\operatorname{id}.
\]
The induced connection on $\operatorname{End}(E_U)$ is unchanged by this
scalar rescaling, because the added connection term is scalar.  Hence
\[
D_h^{0,1}\widetilde\theta
\longmapsto
D_h^{0,1}\widetilde\theta
 -\frac12p^*(\bar\partial_U\partial_U\varphi)\,\operatorname{id}
 =0.
\]
Thus the
horizontal--horizontal component also vanishes.

In summary, all components vanish, and this proves (2).

For later use, we record one additional consequence of the flatness identity
$D_h\Psi=0$.  Its $(2,0)$-part gives
$D_h^{1,0}\widetilde\theta=0$, and hence, together with (2), we have
\begin{equation*}
D_h\widetilde\theta=0.
\end{equation*}

\medskip
\noindent\textbf{Step 2: the polystable case.}
We now use the semisimplicity in Setup~\ref{setup:isomonodromic}.  Write the
fixed isotypic decomposition of the isomonodromic family on $X_U$ as
\begin{align}\label{eq_semisimple_decomp}
(V,\nabla)
=\bigoplus_{i=1}^k (V_i,\nabla_i)\otimes M_i,
\end{align}
where the $(V_i,\nabla_i)$ are pairwise nonisomorphic irreducible flat
families.  Fiberwise non-abelian Hodge theory gives
\begin{align}\label{eq_polys_decomp}
(E_U,\theta_U)
=\bigoplus_{i=1}^k(E_{i,U},\theta_{i,U})\otimes M_i,
\end{align}
with stable factors.  Since the total Dolbeault family is holomorphic,
Proposition~\ref{prop:app-polystable-reduction} implies that every stable-factor
family $(E_{i,U},\theta_{i,U})$ is holomorphic.  Applying Step~1 to each
factor gives a normalized harmonic metric $h_i$ and an absolute Higgs field
$\widetilde\theta_i$ satisfying the three assertions of the theorem and
\(D_{h_i}\widetilde\theta_i=0\).

Fix a positive definite Hermitian metric $H_i$ on each multiplicity space
$M_i$ and put
\[
h:=\bigoplus_i h_i\otimes H_i.
\]
Since the $H_i$ are constant,
\[
\nabla^{\star_h}
=\bigoplus_i\nabla_i^{\star_{h_i}}\otimes\operatorname{id}_{M_i},
\qquad
\widetilde\theta
=\bigoplus_i\widetilde\theta_i\otimes\operatorname{id}_{M_i}.
\]
The relative restriction, holomorphicity, and integrability equations are
block diagonal, so (1)--(3) follow factorwise.  The same calculation also
gives $D_h\widetilde\theta=0$.  
\end{proof}
\begin{rmk}
The metric $h$ constructed above is not merely fiberwise harmonic.  In fact, it
defines a harmonic metric for the absolute Higgs bundle
$(E_U,D_h^{0,1},\widetilde\theta)$ on the total space $X_U$.
Indeed, the above proof shows that $h$ is a solution of the following Hitchin--Simpson equation on the whole space $X_U$:
$$
F_{D_h}
+
[\widetilde\theta,\widetilde\theta^{\star_h}]
=0.
$$
\end{rmk}

\subsection{Absolute lifting over an Artin base}
\label{subsec:finite-order-absolute-lifting}
The $n$-th truncated version of Theorem~\ref{thm:absolute-lifting} also holds. More precisely, we have the following truncated absolute lifting theorem.

\begin{thm}[Finite-order absolute lifting]
\label{thm:truncated-absolute-lifting}
Assume that $\sigma_{\mathrm{Dol}}\circ\gamma$ is holomorphic on
$\operatorname{Spec}A_n$, and let $E_n$ be the induced order $n$ holomorphic
thickening of $E_0$ on $X_n$.  Then the relative Higgs field on $E_n$
admits an absolute lift
\[
\widetilde\Theta_n
\in
H^0\!\left(
X_n,\,
\operatorname{End}(E_n)\otimes\Omega^1_{X_n}
\right)\quad\text{such that}\quad\widetilde\Theta_n\wedge\widetilde\Theta_n=0.
\]
Its image in
$\operatorname{End}(E_n)\otimes\Omega^1_{X_n/A_n}$
is the given relative Higgs field.
\end{thm}

\begin{proof}
We work in the normal gauge of
Theorem~\ref{thm:higher-order-normalization} and suppress the gauge
transformation $\mathscr U$ from the notation.  Thus we write
$$
\theta_t
=
P'_{\eta}(M),
\qquad
M:=\theta_0-\frac12D_{h_0}^{1,0}Y,
$$
and
$$
\bar\partial_t
=
\pi''_{\eta}D_{h_0}
+
P''_{\eta}\!\left(
-\frac12[\theta_0^{\star_{h_0}},Y]
\right).
$$ In particular,
$(E_n,\bar\partial_t,\theta_t)$ is the given holomorphic relative Higgs
bundle.

We use the fixed differentiable trivialization underlying the
Maurer--Cartan class $\eta$ of $X_n$. Define
\begin{equation*}
\widetilde\Theta_n
:=
\theta_t-\frac12\frac{dY}{dt}dt.
\end{equation*}
Its relative restriction is clearly $\theta_t$. We first prove that $\widetilde\Theta_n$ is holomorphic as a section of $
\operatorname{End}(E_n)\otimes\Omega^1_{X_n}.
$

The purely vertical component vanishes because $\theta_t$ is a relative
holomorphic Higgs field.  

For the mixed component, using
$P'_{\eta}(M)=M-\eta(M)$ and the above expression for $\bar\partial_t$,
one obtains
$$
\bigl(\bar\partial_t\widetilde\Theta_n\bigr)_{\mathrm{mix}}
=
P''_{\eta}\!\left(
\frac{d\eta}{dt}(M)
+
\eta\!\left(\frac{dM}{dt}\right)
+
\frac14
\left[
[\theta_0^{\star_{h_0}},Y],
\frac{dY}{dt}
\right]
-
\frac12
\bar\partial_0\!\left(\frac{dY}{dt}\right)
\right)\wedge dt.
$$

Here $\bar\partial_t$ also denotes the induced Dolbeault operator on
$\operatorname{End}(E_n)$-valued forms.

The higher-order identities of Hu--Sun--Yang--Zuo
\cite[Equation~(89) in the proof of Lemma~5.2]{HSYZ} give
\begin{equation*}
\eta(M)
+
\frac14\mathfrak E^{-1}
\bigl[
[\theta_0^{\star_{h_0}},Y],
\mathfrak E Y
\bigr]
=\frac12\bar\partial_0Y,
\end{equation*}
where $\mathfrak E=t\frac{d}{dt}$ and
$\mathfrak E^{-1}$ acts coefficientwise on $(t)$.  Applying
$\mathfrak E$ yields
$$
(\mathfrak E\eta)(M)
+
\eta(\mathfrak E M)
+
\frac14
\bigl[
[\theta_0^{\star_{h_0}},Y],
\mathfrak E Y
\bigr]
=
\frac12\bar\partial_0(\mathfrak E Y).
$$
This is exactly the vanishing of the contraction of the mixed component
of $\bar\partial_t\widetilde\Theta_n$ with $\mathfrak E$.

Now
$$
\Omega^1_{A_n}
=
\bigoplus_{j=0}^{n-1}\mathbb C\,t^jdt,
\qquad
\iota_{\mathfrak E}(t^jdt)=t^{j+1},
$$
so contraction with $\mathfrak E$ is injective on the base-differential
factor.  Hence the mixed component vanishes, and therefore
$$
\widetilde\Theta_n
\in
H^0\!\left(
X_n,\,
\operatorname{End}(E_n)\otimes\Omega^1_{X_n}
\right).
$$

It remains to prove integrability.  The vertical--vertical component of
$\widetilde\Theta_n\wedge\widetilde\Theta_n$ vanishes because
$\theta_t\wedge\theta_t=0$.  The second higher-order identity
\cite[Equation~(90) in the proof of Lemma~5.2]{HSYZ} is
\begin{equation*}
\frac14\mathfrak E^{-1}
\bigl[
D_{h_0}^{1,0}Y,\mathfrak E Y
\bigr]
=\frac12[\theta_0,Y].
\end{equation*}
Applying $\mathfrak E$ gives $
[M,\mathfrak E Y]=0.
$ Since
$$
\iota_{\mathfrak E}
\left(
-\frac12\frac{dY}{dt}\,dt
\right)
=
-\frac12\mathfrak E Y,
$$
the last identity says precisely that the contraction of the mixed
component of
$\widetilde\Theta_n\wedge\widetilde\Theta_n$
with $\mathfrak E$ vanishes.  By the same injectivity argument, the mixed
component itself vanishes.  Finally, the base--base component is zero
because $dt\wedge dt=0$.  Thus $
\widetilde\Theta_n\wedge\widetilde\Theta_n=0.
$\end{proof}

\section{Flatness of spectral one-forms of the isomonodromic family}
\label{sec:spectral-flatness}
The isomonodromic section 
$\sigma_{\mathrm{Dol}}$ defined in Setup~\ref{setup:isomonodromic} is in general only real analytic.  Let $U\subset S$
be an irreducible complex subvariety such that $\sigma_{\mathrm{Dol}}|_U$ is
holomorphic.  We investigate the flatness of its relative spectral one-form.
Locally on $U$, the holomorphic section is
represented by a holomorphic family of Higgs bundles $(E_U,\theta_U)$ on
$X_U:=X\times_S U$ as in Section~\ref{sec:deformation-lifting}.

\begin{defi}[Relative spectral scheme and spectral one-form]\label{def_relative_s_oneform}
 The holomorphic family of Higgs bundles $(E_U,\theta_U)$ induces an \(\mathcal O_{X_U}\)-algebra morphism, denoted by
\[
\Theta_U:\operatorname{Sym}T_{X_U/U}\longrightarrow \mathcal End(E_U).
\]
The \textbf{relative spectral scheme} associated with the family of Higgs bundles $(E_U,\theta_U)$ is
\[
\Sigma_U
:=
\operatorname{Spec}_{X_U}
\left(
\operatorname{Sym}T_{X_U/U}/\ker\Theta_U
\right)
\subset T^*(X_U/U).
\]
Fixing any point in $U$ and denoting it as $0\in U$, its fiber
$\Sigma_0:=\Sigma_U\times_U\{0\}\subset T^*X_0$ is the spectral scheme of
$(E_0,\theta_0)$. Let
\[
\lambda_{\mathrm{can}}
\in
H^0\!\left(
T^*(X_U/U),\,
\Omega^1_{T^*(X_U/U)/U}
\right)
\]
be the tautological relative one-form. If
\(\iota:\Sigma_U\hookrightarrow T^*(X_U/U)\) denotes the inclusion, then
\[
\alpha:=\iota^*\lambda_{\mathrm{can}}
\in
H^0\!\left(
\Sigma_U,\Omega^1_{\Sigma_U/U}
\right)
\]
is called its \textbf{relative spectral one-form}.
\end{defi}

We now define the flatness of the relative spectral one-form \(\alpha\). Assume that
\[
(\Sigma_U)_{\mathrm{red}}=\bigcup_{i=1}^N\Sigma_{U,i}
\]
is the decomposition into irreducible components. For each \(i\), choose a resolution
\[
q_i:\widehat{\Sigma}_{U,i}\longrightarrow \Sigma_{U,i},
\]
and let $\widehat{\pi}_i:\widehat{\Sigma}_{U,i}\longrightarrow U$ be the induced morphism. After shrinking \(U\), we may choose a dense open subset
\(U^\circ\subset U_{\mathrm{reg}}\) such that
\[
\widehat{\pi}_i^\circ:
\widehat{\Sigma}_{U,i}^\circ
:=\widehat{\pi}_i^{-1}(U^\circ)
\longrightarrow U^\circ
\]
is smooth and projective for every \(i\). The smooth projective family \(\widehat{\pi}_i^\circ\) gives a weight one polarized variation of Hodge structures
\[
\mathbb V_i
:=
R^1(\widehat{\pi}_i^\circ)_*\mathbb C,
\]
with associated holomorphic bundle $\mathcal H_i
:=
\mathbb V_i\otimes_{\mathbb C}\mathcal O_{U^\circ}
$ and Gauss--Manin connection
\[
\nabla_i^{\mathrm{GM}}:
\mathcal H_i
\longrightarrow
\mathcal H_i\otimes\Omega^1_{U^\circ}.
\]

Set $\Sigma_{U,i}^\circ:=\Sigma_{U,i}\times_U U^\circ,\ 
q_i^\circ:=q_i|_{\widehat{\Sigma}_{U,i}^\circ},$ and define
\[
\widehat{\alpha}_i
:=
(q_i^\circ)^*\!\left(\alpha|_{\Sigma_{U,i}^\circ}\right)
\in
H^0\!\left(
\widehat{\Sigma}_{U,i}^\circ,
\Omega^1_{\widehat{\Sigma}_{U,i}^\circ/U^\circ}
\right).
\]

Since \(\widehat{\alpha}_i\) restricts to a holomorphic one-form on every fiber, it determines a section
\[
[\widehat{\alpha}_i]
\in
H^0(U^\circ,F^1\mathcal H_i)
\subset
H^0(U^\circ,\mathcal H_i).
\]

\begin{defi}[Flatness of the spectral one-form]\label{def_flat}
The relative spectral one-form \(\alpha\) is said to be \textbf{flat along the irreducible component} \(\Sigma_{U,i}\) if
\[
\nabla_i^{\mathrm{GM}}[\widehat{\alpha}_i]=0.
\]
We say that the spectral one-form \(\alpha\) is \textbf{flat} if it is flat along every irreducible component of
\((\Sigma_U)_{\mathrm{red}}\).
\end{defi}

\begin{rmk}
\begin{enumerate}
\item
The above definition is independent of the choice of the resolution
\(q_i:\widehat{\Sigma}_{U,i}\to\Sigma_{U,i}\).
Indeed, given two resolutions, after possibly shrinking \(U^\circ\),
they admit a common resolution over \(U^\circ\). The induced pullback
maps give isomorphisms of the corresponding \(R^1\)-local systems,
compatible with the Gauss--Manin connections. Since the pullbacks of
the spectral one-form coincide on the common resolution, the definition of flatness is independent of the chosen resolution.

\item
The holomorphicity assumption on the restriction of the section from
Setup~\ref{setup:isomonodromic} is essential for the above definition.
Indeed, without it the corresponding spectral data may vary only real
analytically, and the above PVHS formulation does
not directly apply.
\end{enumerate}
\end{rmk}

In this section, we aim to prove the following flatness theorem for the spectral one-form $\alpha$.
\begin{thm}\label{thm:spectral-flatness}
In the notation of Setup~\ref{setup:isomonodromic}, let $U\subset S$ be an
irreducible complex subvariety such that $\sigma_{\mathrm{Dol}}|_U$ is
holomorphic.  Then its spectral one-form $\alpha$ is flat in the sense of
Definition~\ref{def_flat}.
\end{thm}

\begin{proof}[Proof of Theorem~\ref{thm:spectral-flatness}]
Since flatness is defined componentwise, it suffices to prove the statement
for an arbitrary irreducible component of $(\Sigma_U)_{\mathrm{red}}$.
Fix such a component $\Sigma$. For simplicity, we suppress the index $i$
from the notation and write
\[
q:\widehat{\Sigma}\longrightarrow\Sigma,
\qquad
\widehat{\pi}:\widehat{\Sigma}\longrightarrow U,
\qquad
\widehat{\alpha}
\]
for the corresponding resolution, induced morphism, and spectral one-form.

Fix any point in $U^\circ$, denoted by $0\in U^\circ$. Let
\[
p:T^*(X_U/U)\longrightarrow X_U
\]
be the natural projection and set $\widehat p:=p|_{\Sigma}\circ q:\widehat{\Sigma}\longrightarrow X_U.$ Denote by
\[
\widehat{\Sigma}_0:=\widehat{\pi}^{-1}(0),
\qquad
\widehat p_0:\widehat{\Sigma}_0\longrightarrow X_0
\]
the corresponding fiber and induced morphism.

We work on a dense analytic Zariski open subset $
\widehat{\Sigma}_0^{\mathrm{adm}}\subset\widehat{\Sigma}_0
$ on which \(q\) is an isomorphism and \(\widehat p_0\) is unramified. Set $\widehat E_0:=\widehat p_0^*E_0,$ and define
\[
\widehat\theta_0
:=\widehat p_0^*\theta_0\in H^0\bigl(\widehat{\Sigma}_0,\End(\widehat E_0)\otimes \Omega^1_{\widehat{\Sigma}_0} \bigr).
\]
We also denote $\widehat{\alpha}_0
:=
\widehat{\alpha}|_{\widehat{\Sigma}_0}.$

Fix \(\widehat\xi\in\widehat{\Sigma}_0^{\mathrm{adm}}\) and set $
x:=\widehat p_0(\widehat\xi)$ and $\xi:=q(\widehat\xi)\in T_x^*X_0.
$ Since \(\theta_0\wedge\theta_0=0\), for any $v\in T_{X_0,x},
$ the endomorphisms $
\theta_0(v)$ commute. Hence, for \(\ell\gg0\), the joint generalized eigenspace corresponding to \(\xi\) is
$$
W_{\widehat\xi}
:=
\bigcap_{v\in T_{X_0,x}}
\ker\!\left(
\theta_0(v)-\xi(v)\operatorname{id}
\right)^\ell.
$$
After removing a proper analytic subset from
\(\widehat{\Sigma}_0^{\mathrm{adm}}\), and retaining the same notation for the resulting dense open subset, we may assume that these spaces $W_{\widehat\xi}$ have locally constant dimension and therefore glue to a holomorphic subbundle $
W\subset\widehat E_0.$

Let $W'$ be the sum of the remaining blocks. Then $\widehat E_0=W\oplus W'$ on \(\widehat{\Sigma}_0^{\mathrm{adm}}\). Let
\[
P\in\mathcal End(\widehat E_0)
\]
be the corresponding holomorphic projection bundle map of $W$, so that
\[
P^2=P,
\qquad
\operatorname{Im}P=W,
\qquad
\ker P=W'.
\]

By the definition of the generalized eigenspace,
\[
\widehat\theta_0|_W
=
\widehat{\alpha}_0\,\operatorname{id}_W+N,
\]
where $N$ is an $\operatorname{End}(W)$-valued nilpotent one-form. Hence $\operatorname{tr}_W(N)=0.$ Taking traces gives
\[
\operatorname{tr}_W(\widehat\theta_0)
=
m\,\widehat{\alpha}_0,\qquad\text{ where } m:=\operatorname{rank}W.
\]
By definition, we have $\operatorname{tr}(P\widehat\theta_0)
=
\operatorname{tr}_W(\widehat\theta_0).$ Therefore
\begin{equation}\label{eq:normalized-trace-spectral-form}
\widehat{\alpha}_0
=
\frac{1}{m}\operatorname{tr}(P\widehat\theta_0)
=
\frac{1}{m}\operatorname{tr}_W(\widehat\theta_0).
\end{equation}

Let $v\in T^{1,0}_0U^\circ$ and choose a holomorphic disc through $0$
tangent to $v$, with coordinate $t$.  The above argument holds for $X_t$ over $t$ similarly. Denote the
resulting projector by $P_t$ and the pulled-back first harmonic metric variation $g_1$
by $\widehat g_1$.  By 
\eqref{eq:g1-holomorphicity} and \eqref{eq:first-order-higgs}, we have
\begin{equation}\label{eq:variation-pulledback-higgs}
\left.\frac{d\widehat\theta_t}{dt}\right|_{t=0}
=-\frac12D_{\widehat h_0}\widehat g_1,
\end{equation}
where $D_{\widehat h_0}$ is the pullback Chern connection.

Differentiating $P_t^2=P_t$ at $t=0$ gives
\[
P\dot P_0P=0,
\qquad
(1-P)\dot P_0(1-P)=0.
\]
Thus $\dot P_0$ is off diagonal, whereas $\widehat\theta_0$ is block
diagonal.  Consequently $\operatorname{tr}(\dot P_0\widehat\theta_0)=0$.

Differentiating \eqref{eq:normalized-trace-spectral-form} and using
\eqref{eq:variation-pulledback-higgs}, we obtain on the admissible locus
\begin{equation}\label{eq:variation-spectral-form}
\left.\frac{d\widehat\alpha_t}{dt}\right|_{t=0}
=-\frac{1}{2m}
\operatorname{tr}_W(D_{\widehat h_0}\widehat g_1).
\end{equation}
  By \eqref{eq:g1-holomorphicity}, we have $[\widehat\theta_0,\widehat g_1]=0$, so
$\widehat g_1$ preserves $W$, commutes with $P$, and is block diagonal
with respect to $\widehat E_0=W\oplus W'$. Hence
\[
\operatorname{tr}_W(D_{\widehat h_0}\widehat g_1)
=d\operatorname{tr}_W(\widehat g_1).
\]
Because $P$ is single-valued on the admissible locus,
$\operatorname{tr}_W(\widehat g_1)=\operatorname{tr}(P\widehat g_1)$
is a single-valued smooth function there.  Equation
\eqref{eq:variation-spectral-form} is therefore exact.

Let $[\gamma]\in H_1(\widehat\Sigma_0,\mathbb Z)$.  The complement of
$\widehat\Sigma_0^{\mathrm{adm}}$ has positive complex codimension, so
$[\gamma]$ has a smooth representative contained in the admissible
locus.  Transport it by a smooth horizontal trivialization to cycles
$\gamma_t$.  Then
\[
\left.\frac{d}{dt}\right|_{t=0}
\int_{\gamma_t}\widehat\alpha_t
=-\frac{1}{2m}\int_\gamma
d\operatorname{tr}_W(\widehat g_1)=0.
\]
Changing the horizontal trivialization changes the ordinary derivative
by a Lie derivative.  Since $\widehat\alpha_0$ is a holomorphic one-form
on the smooth projective fiber, it is closed, and Cartan's formula shows
that this change is again exact.  Thus every Gauss--Manin transported
period has zero derivative in the direction $v$.

The point $0\in U^\circ$, the tangent direction $v$, and the irreducible
component $\Sigma$ were arbitrary.  Hence $\nabla^{\mathrm{GM}}[\widehat\alpha_i]=0$ for every component, proving the theorem.
\end{proof}
\begin{rmk}
Locally on \(U\), the conclusion of Theorem~\ref{thm:spectral-flatness} admits a stronger interpretation.  Indeed, over a small contractible neighborhood \(B\subset U\), Theorem~\ref{thm:absolute-lifting} provides an absolute lift of the relative Higgs bundle.  The spectral one-form $\widetilde\alpha$ of $\widetilde\theta$ is also globally defined. By a similar argument as in the proof of Theorem~\ref{thm:spectral-flatness}, we may prove that $\widetilde\alpha$ is closed. Moreover, the relative restriction of \(\widetilde\alpha\) is precisely \(\widehat\alpha\), which gives another explanation for its Gauss--Manin flatness.
\end{rmk}

\section{Deformation rigidity of nilpotent Higgs bundles}
\label{sec:nilpotent-rigidity}
We now apply the flatness theorem to the nilpotent locus.  The main
statement of this section is the following.

\begin{thm}
\label{thm:nilpotence-rigidity}
Under
 Setup~\ref{setup:isomonodromic}, let $(E_0,\theta_0)$ represent
$\sigma_{\mathrm{Dol}}(0)$.  Let $U$ be an irreducible complex analytic subvariety through
$0\in S$.  Assume that
$\sigma_{\mathrm{Dol}}|_U$ is holomorphic.  If $\theta_0$ is nilpotent, then
$\sigma_{\mathrm{Dol}}(u)$ is represented by a nilpotent Higgs bundle for
every $u\in U$.
\end{thm}

We shall use the following nilpotent criteria to prove Theorem~\ref{thm:nilpotence-rigidity} together with Theorem~\ref{thm:spectral-flatness}.

\begin{lem}\label{lem:spectral-criterion}
Let $(E,\theta)$ be an integrable Higgs bundle on a smooth complex
manifold $Y$.  The following conditions are equivalent:
\begin{enumerate}
\item $\theta$ is nilpotent;
\item the spectral one-form vanishes on the resolution of every reduced spectral component;
\item all Hitchin invariant forms vanish.
\end{enumerate}
\end{lem}

We first establish a technical lemma showing that a Gauss--Manin flat family of holomorphic one-forms vanishes identically if it vanishes on one fiber.

\begin{lem}\label{lem:hodge-norm-specialization}
Let $q:Z\to\Delta$ be a projective holomorphic map with $Z$ smooth and
with every irreducible component dominating the disc $\Delta$.  Assume that $q$
is smooth over $\Delta^*$.  Let
$\alpha\in H^0(Z,\Omega^1_{Z/\Delta})$, and suppose that there is a
smooth absolute one-form $A\in\mathcal A^1(Z)$ whose restriction to
every fiber is $\alpha_t$.  If
\[
\nabla^{\mathrm{GM}}[\alpha_t]=0\quad(t\ne0),
\qquad
\alpha|_{(Z_0)_{\mathrm{red}}}=0,
\]
then, after shrinking $\Delta$, one has $\alpha_t=0$ for every
$t\in\Delta^*$.
\end{lem}

\begin{proof}
Since $Z$ is smooth and no irreducible component is contained in $Z_0$,
the parameter $t$ is a non-zero divisor in $\mathcal O_Z$.  Thus $q$ is
flat over $\Delta$.  The fibers consequently form an
analytic family of compact cycles, and their integration currents vary
continuously \cite{Barlet}.

Choose a smooth \(d\)-closed real \((1,1)\)-form \(\omega\) on \(Z\) that is relatively Kähler over \(\Delta\); that is, its restriction $
\omega_t:=\omega|_{Z_t}
$ is a Kähler form on every nearby smooth fiber \(Z_t\). Put $d=\dim_{\mathbb C}Z_t$.  For $t\ne0$
define
\[
I(t):=\sqrt{-1}\int_{Z_t}
\alpha_t\wedge\overline{\alpha_t}\wedge\omega_t^{d-1}.
\]
The classes $[\alpha_t]$, $[\overline{\alpha_t}]$, and $[\omega_t]$
are Gauss--Manin flat, so $I(t)$ is locally constant on the connected
punctured disc. Set
\[
\beta:=\sqrt{-1}A\wedge\overline A\wedge\omega^{d-1}.
\]
The central vanishing hypothesis says that $A$ vanishes on tangent
directions at every smooth point of every irreducible component of
$(Z_0)_{\mathrm{red}}$.  Hence
$\int_{[Z_0]}\beta=0$.  Continuity of the fiber cycles gives
\[
\lim_{t\to0}I(t)
=\lim_{t\to0}\int_{[Z_t]}\beta
=\int_{[Z_0]}\beta=0.
\]
Therefore $I(t)=0$ for $t\ne0$.  K\"ahler positivity for a holomorphic
one-form implies $\alpha_t=0$ on every smooth nearby fiber.
\end{proof}

\begin{rmk}\label{rmk:absolute-lift-spectral-form}
The absolute-lift hypothesis in
Lemma~\ref{lem:hodge-norm-specialization} is automatic for the relative spectral one-form defined in Definition~\ref{def_relative_s_oneform}.  Choose a smooth horizontal splitting for the smooth
family $X_\Delta/\Delta$, lift the tautological one-form on
$T^*(X_\Delta/\Delta)$ to a smooth absolute one-form, and pull it back to the
resolved spectral component.
\end{rmk}

\begin{proof}[Proof of Theorem~\ref{thm:nilpotence-rigidity}]
Let
\[
N:=\{u\in U\mid \sigma_{\mathrm{Dol}}(u)
\text{ is represented by a nilpotent Higgs bundle}\}.
\]
By Lemma~\ref{lem:spectral-criterion}, nilpotence is equivalent to the
vanishing of all positive-degree Hitchin invariants. Since the Hitchin map
and $\sigma_{\mathrm{Dol}}|_U$ are holomorphic, $N$ is a closed analytic
subset of $U$.

We claim that $N$ is open. Suppose otherwise, and let $p\in N$ be a point
at which $N$ is not a neighborhood. By complex-analytic curve selection
\cite{Chirka}, after shrinking around $p$ there exists a holomorphic disc
\[
g:(\Delta,0)\longrightarrow(U,p)
\]
such that $g(t)\notin N$ for every sufficiently small $t\ne0$. Pull back
the smooth projective family and the isomonodromic section along $g$. Since
$\Delta$ is simply connected, the fixed semisimple isomonodromic local
system admits the decomposition
\eqref{eq_semisimple_decomp}--\eqref{eq_polys_decomp} on the disc. By
Proposition~\ref{prop:app-polystable-reduction}, after shrinking $\Delta$,
each stable factor defines a holomorphic isomonodromic Higgs family.

Fix one stable factor and a reduced irreducible component $\Sigma$ of its
relative spectral scheme which dominates $\Delta$. Choose a projective resolution $q:\widehat \Sigma\longrightarrow\Sigma.$ After shrinking $\Delta$, the composed map
$\widehat \Sigma\to\Delta$ is smooth over $\Delta^*$. Let $\widehat\alpha$ be
the pullback of the spectral one-form. For $t\ne0$,
Theorem~\ref{thm:spectral-flatness} gives the Gauss-Manin flatness of $[\widehat\alpha_t]$. Since the central Higgs field of the stable factor is nilpotent,
Lemma~\ref{lem:spectral-criterion} gives
\[
\widehat\alpha|_{(\widehat \Sigma_0)_{\mathrm{red}}}=0.
\]
The smooth absolute-lift hypothesis is automatic by
Remark~\ref{rmk:absolute-lift-spectral-form}. Therefore
Lemma~\ref{lem:hodge-norm-specialization} yields
\[
\widehat\alpha_t=0
\]
for every sufficiently small $t$. Applying
Lemma~\ref{lem:spectral-criterion} again, the stable factor remains
nilpotent near $0$. The same argument applies to every stable factor, and
\eqref{eq_polys_decomp} then shows that the whole Higgs field is nilpotent
for every sufficiently small $t$. This contradicts the choice of $g$.
Thus $N$ is open.

Since $N$ is both open and closed in $U$ and
$0\in N$, we conclude that $N=U$.
\end{proof}

\appendix
\section{Polystable reduction for holomorphic isomonodromic families}
\label{app:polystable-reduction}

In this appendix, we work on a simply connected analytic base
$U$. Let $p:X_U\longrightarrow U$ be a smooth projective family, and let $(V,\nabla)$ be a semisimple
isomonodromic family.  We use
throughout the convention in
\eqref{eq_semisimple_decomp}--\eqref{eq_polys_decomp}.  Thus the
$(V_i,\nabla_i)$ are pairwise nonisomorphic irreducible flat families and the stable factors $(E_{i,u},\theta_{i,u})$ are
pairwise nonisomorphic.  We denote the corresponding real
analytic isomonodromic Higgs families by $(E_U,\theta_U)$ and
$(E_{i,U},\theta_{i,U})$. We aim to show that the isomonodromic family $(E_U,\theta_U)$ is holomorphic if and only if each factor $(E_{i,U},\theta_{i,U})$ is a holomorphic family. We give two lemmas before proving this proposition.

\medskip

For a Higgs bundle $(F,\vartheta)$ on a smooth projective variety $Y$, write
\[
C^\bullet(F,\vartheta)
:=
\bigl(
\operatorname{End}F
\xrightarrow{\operatorname{ad}(\vartheta)}
\operatorname{End}F\otimes\Omega_Y^1
\xrightarrow{\operatorname{ad}(\vartheta)}\cdots
\bigr)
\]
for its Higgs deformation complex.  Fix $u\in U$ and abbreviate
\[
C_u^\bullet:=C^\bullet(E_u,\theta_u),
\qquad
C_{i,u}^\bullet:=C^\bullet(E_{i,u},\theta_{i,u}).
\]
The decomposition \eqref{eq_polys_decomp} gives
\begin{equation*}
\operatorname{End}E_u
=
\bigoplus_{i,j}
\operatorname{Hom}(E_{j,u},E_{i,u})
\otimes\operatorname{Hom}(M_j,M_i),
\end{equation*}
and the differential $\operatorname{ad}(\theta_u)$ preserves every
$(i,j)$-block.

Let $
m_i:=\dim M_i,$ and let $\delta_{ij}$ denote the Kronecker delta, namely, $
\delta_{ij}=
\begin{cases}
1, & i=j,\\
0, & i\neq j.
\end{cases}
$
\begin{lem}
\label{lem:app-isotypic-projection}

For every $i$, there are natural morphisms of complexes
$$
\jmath_i:C_{i,u}^\bullet\longrightarrow C_u^\bullet,
\qquad
\mathsf p_i:C_u^\bullet\longrightarrow C_{i,u}^\bullet,
$$
such that $
\mathsf p_i\circ\jmath_j=\delta_{ij}\operatorname{id}.
$

Explicitly, $\jmath_i$ is the inclusion into the $(i,i)$-block given by $
a\longmapsto a\otimes\operatorname{id}_{M_i},
$
whereas $\mathsf p_i$ vanishes on all other blocks and, on the $(i,i)$-block, is given by
$$
\mathsf p_i
=
\operatorname{id}\otimes
\frac{1}{m_i}\operatorname{tr}_{M_i}.
$$
Consequently, $\jmath_i$ and $\mathsf p_i$ induce split morphisms on hypercohomology.
\end{lem}

\begin{proof}
It is enough to check compatibility with the Higgs differential.  On the
$(i,i)$-block one has $\operatorname{ad}(\theta_u)
=
\operatorname{ad}(\theta_{i,u})\otimes\operatorname{id}_{\operatorname{End}M_i}$ because $\theta_u$ restricts there to
$\theta_{i,u}\otimes\operatorname{id}_{M_i}$.  Therefore
\[
\mathsf p_i\bigl([\theta_u,A]\bigr)
=
[\theta_{i,u},\mathsf p_i(A)].
\]
The same calculation shows that $\jmath_i$ is a chain map.  The identity
$\mathsf p_i\circ\jmath_j=\delta_{ij}\operatorname{id}$ follows from
$\operatorname{tr}(\operatorname{id}_{M_i})=m_i$.
\end{proof}

Let $\eta\in\mathcal A^{0,1}(T_Y)$ represent a first-order deformation of the
complex structure of $Y$.  The Higgs field defines the standard morphism of
complexes
\[
(T_Y,0)\longrightarrow C^\bullet(F,\vartheta),
\]
and hence a map
\[
\vartheta_*:H^1(Y,T_Y)
\longrightarrow
\mathbb H^1\bigl(Y,C^\bullet(F,\vartheta)\bigr).
\]
In the Dolbeault resolution its value is represented by
\begin{equation}\label{eq:app-theta-obstruction}
\vartheta_*([\eta])
=
\bigl[(0,\eta\lrcorner\vartheta)\bigr].
\end{equation}
This is the first-order class appearing in the Cauchy--Riemann criterion for
isomonodromic Higgs deformations; see Hu--Sun--Yang--Zuo
\cite[Proposition~3.5]{HSYZ}. 

\begin{lem}
\label{lem:app-obstruction-splitting}
Let $v\in T_u^{1,0}U$, and let $[\eta_v]\in H^1(X_u,T_{X_u})$ be the
Kodaira--Spencer class in the direction $v$.  Then
\begin{equation}\label{eq:app-obstruction-projection}
(\mathsf p_i)_*\bigl(\theta_{u,*}([\eta_v])\bigr)
=
\theta_{i,u,*}([\eta_v])
\end{equation}
for every $i$.  More precisely,
\begin{equation}\label{eq:app-obstruction-directsum}
\theta_{u,*}([\eta_v])
=
\sum_{i=1}^k
(\jmath_i)_*\bigl(\theta_{i,u,*}([\eta_v])\bigr).
\end{equation}
In particular,
\[
\theta_{u,*}([\eta_v])=0
\quad\Longleftrightarrow\quad
\theta_{i,u,*}([\eta_v])=0
\ \text{for every }i.
\]
\end{lem}

\begin{proof}
By the Higgs decomposition \eqref{eq_polys_decomp}, we have
\[
\eta_v\lrcorner\theta_u
=
\bigoplus_i
(\eta_v\lrcorner\theta_{i,u})\otimes\operatorname{id}_{M_i}.
\]
Applying $\mathsf p_i$ and using the normalized trace gives
\[
\mathsf p_i(\eta_v\lrcorner\theta_u)
=
\eta_v\lrcorner\theta_{i,u}.
\]
Together with \eqref{eq:app-theta-obstruction}, this proves
\eqref{eq:app-obstruction-projection}; the direct-sum \eqref{eq:app-obstruction-directsum} follows from the explicit block
expression.  The last assertion follows because the maps $\mathsf p_i$ split
the diagonal inclusions $\jmath_i$.
\end{proof}

This Lemma~\ref{lem:app-obstruction-splitting} and the first order holomorphicity criterion prove the following proposition.

\begin{prop}
\label{prop:app-polystable-reduction}
In the situation of \eqref{eq_semisimple_decomp}--\eqref{eq_polys_decomp},
the following conditions are equivalent:
\begin{enumerate}
\item the real analytic isomonodromic Higgs family $(E_U,\theta_U)$ is
holomorphic;
\item for every $i$, the stable isomonodromic Higgs family
$(E_{i,U},\theta_{i,U})$ is holomorphic;
\item for every $u\in U$, every $v\in T_u^{1,0}U$, and every $i$, one has $\theta_{i,u,*}([\eta_v])=0.$
\end{enumerate}
\end{prop}

\thispagestyle{empty}

\end{document}